\documentclass[a4paper]{amsart}

\usepackage{amssymb}
\usepackage{amsthm}  
\usepackage{amsmath} 
\usepackage[all]{xypic}
\usepackage{url}

\swapnumbers

\newcommand{\om}{\omega}
\newcommand{\Om}{\Omega}
\newcommand{\ka}{\kappa}
\newcommand{\si}{\sigma}

\newcommand{\ba}{\mathcal{G}}  
\newcommand{\va}{\varphi}
\newcommand{\fg}{\mathfrak g}
\newcommand{\fp}{\mathfrak p}
\newcommand{\Ad}{{\rm Ad}}
\newcommand{\na}{\nabla}
\newcommand{\Rho}{{\mbox{\sf P}}}
\newcommand{\U}{\Upsilon}
\newcommand{\gr}{{\operatorname{gr}}}
\newcommand{\Gr}{{\operatorname{Gr}}}

\newtheorem*{prop*}{Proposition}

\newtheorem*{thm*}{Theorem}
\newtheorem{lemma}[subsection]{Lemma}
\newtheorem*{lemma*}{Lemma}

\newtheorem*{cor*}{Corollary}

\newtheorem*{rem*}{Remark}
\theoremstyle{definition}
\newtheorem*{def*}{Definition}

\theoremstyle{remark}

\def\R{\mathbb{R}}

\def\x{\times}
\def\o{\circ}
\def\al{\alpha}
\def\fq{\mathfrak q}
\def\fh{\mathfrak h}
\def\G{\mathcal{G}}
\def\P{\mathcal{P}}
\def\A{\mathcal{A}}
\def\L{\mathcal{L}}
\def\D{\mathcal{D}}
\def\tG{\tilde G}
\def\tP{\tilde P}

\def\tg{\tilde\fg}
\def\tp{\tilde\fp}

\def\tom{\tilde\om}
\def\pmat#1{\begin{pmatrix}#1\end{pmatrix}}
\def\id{\operatorname{id}}
\def\Fl{\operatorname{Fl}}

\begin{document}
\title{Remarks on Grassmannian Symmetric Spaces}
\author{Lenka Zalabov\'a and Vojt\v ech \v Z\'adn\'ik}
\address{
Tomas Bata University, Zl\'in and Masaryk University, Brno, Czech Republic }
\email{zalabova@math.muni.cz and zadnik@math.muni.cz}


\thanks{First author supported at different times by the Eduard \v Cech Center, project
nr.\ LC505, and the ESI Junior Fellows program; 
second author supported by the grant nr.\ 201/06/P379 of the Grant Agency of 
Czech Republic.}

\begin{abstract}
The classical concept of affine locally symmetric spaces allows a
generalization for various geometric structures on a smooth manifold. 
We remind the notion of symmetry for parabolic geometries and we summarize 
the known facts for $|1|$--graded parabolic geometries and for almost Grassmannian 
structures, in particular. 
As an application of two general constructions with parabolic geometries, 
we present an example of non--flat Grassmannian symmetric space.
Next we observe there is a distinguished torsion--free affine connection preserving the
Grassmannian structure so that, with respect to this connection, the Grassmannian symmetric
space is an affine symmetric space in the classical sense.
\end{abstract}

\maketitle


\section{Introduction}
Affine (locally) symmetric spaces present a very classical topic in
differential geometry, see e.g.\ to \cite[chapter XI]{KN2} for all details.
In particular,  a smooth manifold $M$ with an affine connection is called affine 
locally symmetric space 
if for each point $x\in M$ there is a local symmetry centered at $x$, i.e. a locally 
defined affine transformation $s_x$ such that $x$ is an isolated fixed point
of $s_x$ and  $T_xs_x=-\id$.
The notion of local symmetry is easily modified for various geometric
structures and there are attempts to understand these generalizations.
For instance, there is a lot known on the so--called projectively symmetric
spaces, see e.g.\ \cite{P} and references therein.
In the framework of parabolic geometries, the general definition of local symmetry
fits nicely especially for $|1|$--graded parabolic geometries which includes 
projective, conformal, and almost Grassmannian geometries as particular
examples.

We open the article with a review of basic ideas and concepts for parabolic
geometries, including the notion of Weyl structures and normal
coordinates. 
Then we introduce local symmetries for general parabolic geometries  and
provide their characterization in homogeneous model, Proposition \ref{2.4}.
There is a couple of general results for symmetric $|1|$--graded parabolic
geometries \cite{ja,dis} which we recover in section \ref{3} in a slightly
improved way. 
In particular, for a local symmetry centered at $x$, Theorem \ref{3.2}
provides an existence of a torsion--free Weyl connection, locally defined in a
neighborhood of $x$, which is invariant with respect to the symmetry and
whose Rho--tensor vanishes at $x$.
Then we focus on Grassmannian (locally) symmetric spaces, i.e.\ 
almost Grassmannian structures allowing a (local) symmetry at each point.
It turns out that the model Grassmannian structure is always symmetric and
a non--flat almost Grassmannian structure of type $(p,q)$ may be locally
symmetric only if $p$ or $q$ is 2.
In the latter case, the possible local symmetries at a point are heavily restricted, 
see \ref{3.4}.

The rest of the paper is devoted to the description of an example of non--flat
Grassmannian locally symmetric space of type $(2,q)$.
This appears as the space of chains of the homogeneous model of a parabolic
contact geometry, Theorem \ref{th1}.
(The notion of chains here generalizes the Chern--Moser chains on CR manifolds 
of hypersurface type.)
The example comes as an application of some general constructions from
\cite{CZ} and \cite{Cor}, dealing with the parabolic geometry associated to the 
path geometry of chains. 
The necessary background for a comfortable understanding of the result is 
presented in \ref{4.2} and \ref{4.3}.
In addition, the constructed space is globally symmetric and there is a
torsion--free affine connection preserving the Grassmannian structure which
is invariant with respect to some distinguished symmetries, Theorem \ref{4.5}.
Hence, we end up with an affine symmetric space with a compatible 
Grassmannian structure, Corollary \ref{4.5}.

\subsection*{Acknowledgements} 
We would like to mention the discussions with Andreas \v Cap, Boris
Doubrov, and Jan Slov\'ak, as well as the remarks by the anonymous referee, 
which were very helpful during the work on this paper. 

\section{Parabolic geometries, Weyl structures, and symmetries} \label{2}
In this section, we remind definitions and basic facts on Cartan
geometries, Weyl structures and symmetries for parabolic geometries.
We primarily refer to \cite{C,WS,parabook} for more intimate and comprehensive
introduction to parabolic geometries, the subsection dealing with symmetries 
is based on \cite{ja}.

\subsection{Definitions} 			\label{2.1}
Let $G$ be a Lie group, $P\subset G$ its Lie subgroup, and $\fp\subset\fg$ the
corresponding Lie algebras. 
A \emph{Cartan geometry} of type $(G,P)$ on a smooth manifold $M$ 
is a couple $(\G\to M,\om)$ consisting of a
principal  $P$--bundle $\ba \rightarrow M$ together with a 
one--form $\om \in \Omega^1(\ba,\fg)$, which is $P$--equivariant, 
reproduces the fundamental vector fields and induces a linear isomorphism
$T_u\ba\cong\fg$ for each $u \in \ba$. 
The one--form $\om$ is called the \emph{Cartan connection}.
The \emph{curvature} of  Cartan geometry is defined as 
$K:= d\om+\frac12[\om,\om]$, which is a two--form on $\G$ with values in $\fg$.
Easily, the $P$--bundle $G\rightarrow G/P$ with the (left) Maurer--Cartan form
$\mu\in\Om^1(G,\fg)$ form a Cartan geometry of type $(G,P)$ with vanishing
curvature, which we call the  \emph{homogeneous model}.

\emph{Parabolic geometry} is a Cartan geometry $(\ba \rightarrow M, \om)$  of 
type $(G,P)$, where $G$ is a semisimple Lie group and $P$ its parabolic 
subgroup. The Lie algebra $\fg$ of the Lie group $G$ is 
equipped (up to the choice of Levi factor $\fg_0$ in $\fp$) with the grading 
of the form $\fg=\fg_{-k}\oplus \dots \oplus \fg_{0} \oplus \dots \oplus\fg_{k}$ 
such that the Lie algebra $\fp$ of $P$ is $\fp = \fg_0 \oplus \dots \oplus \fg_k$. 
Suppose the grading of $\fg$ is fixed and further denote $\fg_{-}:=\fg_{-k} 
\oplus \dots \oplus \fg_{-1}$ and $\fp_+ := \fg_1\oplus\dots\oplus\fg_k$.
Parabolic geometry corresponding to the grading of length $k$ is called \emph{$|k|$--graded}. 
By $G_0$ we denote the subgroup in $P$, with the Lie algebra $\fg_0$,
consisting of  all elements in $P$ whose adjoint action preserves the grading of $\fg$. 
Next, defining $P_+:=\exp\fp_+$, we get $P/P_+=G_0$ and $P=G_0\rtimes P_+$.

The grading of $\fg$ induces a $P$--invariant filtration
$ \fg=\fg^{-k}\supset \fg^{-k+1} \supset \dots \supset \fg^{k}=\fg_k$,
where $\fg^{i}:=\fg_i \oplus \dots \oplus \fg_k$.
This gives rise to a filtration of the tangent bundle $TM$ as follows.
The Cartan connection $\om$ provides an identification $TM\cong\G\x_P(\fg/\fp)$
where the action of $P$ on $\fg/\fp$ is induced by the adjoint representation.
Hence each $P$--invariant subspace $\fg^{-i}/\fp\subset\fg/\fp$ defines the subbundle 
$T^{-i}M:=\G\x_P(\fg^{-i}/\fp)$ in $TM$, so we obtain the filtration 
$TM=T^{-k}M\supset\dots\supset T^{-1}M$.
Alternatively, the filtration is described using the \emph{adjoint tractor
bundle} which is the natural bundle $\A M:=\G\x_P\fg$ corresponding to the 
(restriction of) adjoint representation of $G$ on $\fg$.
The filtration of $\fg$ induces a filtration 
$\A M=\A^{-k}M\supset\dots\supset\A^0M\supset\dots\supset\A^kM$ so that
$T^{-i}M\cong\A^{-i}M/\A^0M$, in particular, $TM\cong\A M/\A^0M$.
Next by $\gr(TM)$ we denote the associated graded bundle
$\gr(TM)=\gr_{-k}(TM)\oplus\dots\oplus\gr_{-1}(TM)$, where
$\gr_{-i}(TM):=T^{-i}M/T^{-i+1}M$ is the associated bundle to $\G$ with
the standard fiber $\fg^{-i}/\fg^{-i+1}$ which is isomorphic to $\fg_{-i}$ as
a $G_0$--module.
Since $P_+\subset P$ acts freely on $\G$, the quotient $\G/P_+=:\G_0$ is a
principal bundle over $M$ with the structure group $P/P_+=G_0$.
Hence $\gr(TM)\cong\G_0\x_{G_0}\fg_-$ and the Lie bracket on $\fg_-$ induces
an algebraic bracket on $\gr(TM)$.

By definition, the curvature $K\in\Om^2(\G,\fg)$ of a parabolic geometry is
strictly horizontal and $P$--equivariant, hence it is fully described by a 
two--form on $M$ with values in $\G\x_P\fg=\A M$, which is denoted by $\ka$.
Note that by the same symbol we also denote the corresponding frame form,
which is a $P$--equivariant map $\ba \rightarrow 
\wedge^2 (\fg/\fp)^* \otimes \fg$, the so--called \emph{curvature function}.
The Killing form on $\fg$ provides an identification $(\fg/\fp)^*\cong\fp_+$,
hence the curvature function is viewed as having values in
$\wedge^2\fp_+\otimes\fg$.
The grading of $\fg$ brings a grading to this space and parabolic geometry
is called \emph{regular} if the curvature function  has values in the part of 
positive homogeneity.
The parabolic geometry is regular if and only if the algebraic bracket on
$\gr(TM)$ above coincides with the \emph{Levi bracket}, which is the natural
bracket induced by the Lie bracket of vector fields.
Next, the parabolic geometry is called \emph{torsion--free} if $\ka$ has values in 
$\wedge^2\fp_+\otimes \fp$; note that torsion--free parabolic geometry is
automatically regular.
Altogether, 
for a regular parabolic geometry, there is an underlying structure on $M$ consisting of
a filtration of the tangent bundle (which is compatible with the Lie bracket
of vector fields) and a reduction of the structure group of $\gr(TM)$ to the
subgroup $G_0$.

The correspondence between regular parabolic geometries of specified type and
the underlying structures can be made bijective (up to isomorphism) provided one impose 
some normalization condition:
The parabolic geometry is called \emph{normal} if $\partial^* \circ \ka =0$, where
$\partial^*:\wedge^2\fp_+\otimes\fg\to\fp_+\otimes\fg$ is the differential in
the standard complex computing the homology $H_*(\fp_+,\fg)$ of $\fp_+$ with coefficients in
$\fg$.
Dealing with regular normal parabolic geometries, there is the notion of
\emph{harmonic curvature} $\ka_H$, which is the composition of $\ka$ with the
natural projection $\ker(\partial^*)\to H_2(\fp_+,\fg)$.
By definition, $\ka_H$ is a section of $\G\x_P H_2(\fp_+,\fg)$ and, since $P_+$
acts trivially on $H_*(\fp_+,\fg)$, it can be interpreted in terms of the
underlying structure.
The main issue is that the harmonic curvature $\ka_H$ is much simpler object than the
curvature $\ka$, however, still involving the whole information about $\ka$.
Note that, as a $G_0$--module, each $H_j(\fp_+,\fg)$  is isomorphic 
to $\ker\square\subset\ker\partial^*\subset\wedge^j\fp_+\otimes\fg$, the kernel of the 
Kostant Laplacian.
As a consequence of \cite[Corollary 4.10]{CS}, which is an application of generalized 
Bianchi identity, we conclude:
\begin{lemma*}
  Let $\ka$ and $\ka_H$ be the Cartan curvature and the harmonic curvature
  of a regular normal parabolic geometry of type $(G,P)$.
  Then the lowest non--zero homogeneous component of $\ka$ has values in
  $\ker\square\subset\wedge^2\fp_+\otimes\fg$, i.e.\ it coincides with the
  corresponding homogeneous component of $\ka_H$.
  In particular, $\ka_H=0$ if and only if  $\ka=0$.
\end{lemma*}
If $\ka=0$, the parabolic geometry is called \emph{flat} (or \emph{locally
flat}).
Flat parabolic geometry of type $(G,P)$ is locally isomorphic to the homogeneous model 
$(G\to G/P,\mu)$.	

\subsection{Examples}			\label{2.2}
Here we focus on two classes of parabolic geometries which are often mentioned in the
sequel:

(1) 
An important family of examples is formed by $|1|$--graded parabolic geometries.
Any $|1|$--graded parabolic geometry is trivially regular and 
the main feature of any such geometry is that
the tangent bundle $TM$ has not got any nontrivial natural filtration.
Hence (up to one exception) the underlying structure on $M$ is just a classical first 
order $G_0$--structure.  
All the section \ref{3} deals with $|1|$--graded parabolic geometries, 
with almost Grassmannian structures in particular.

(2)
Another interesting examples are the \emph{parabolic contact geometries}, which 
are $|2|$--graded parabolic geometries with underlying contact structure.
Parabolic contact geometry corresponds to a \emph{contact grading} of a simple Lie 
algebra $\fg$, which is a grading 
$\fg=\fg_{-2}\oplus\fg_{-1}\oplus\fg_0\oplus\fg_1\oplus\fg_2$ such
that $\fg_{-2}$ is one dimensional and the Lie bracket $[\ ,\ ]:\fg_{-1}\x\fg_{-1}
\to\fg_{-2}$ is non--degenerate.
The filtration of $TM$ looks like $TM=T^{-2}M\supset T^{-1}M$ so that
$\D:=T^{-1}M$ is the contact distribution.
For regular parabolic contact geometries, the Levi bracket $\L:\D\x\D\to TM/\D$
is non--degenerate and the reduction of $\gr(TM)=(TM/\D)\oplus \D$ to the structure 
group $G_0$ corresponds to an additional structure on $\D$.

The best known examples of parabolic contact geometries are non--degenerate
partially integrable almost CR structures of hypersurface type where the additional
structure on $\D$ is an almost complex structure.
Another examples are the Lagrangean contact structures which are introduced in
\ref{4.4} in some detail.

\subsection{Weyl structures and connections}		\label{2.5}
Let $(\G\to M,\om)$ be a parabolic geometry of type $(G,P)$, 
let $\G_0=\G/P_+$ be the underlying $G_0$--bundle as in \ref{2.1}, and let
$\pi:\G\to\G_0$ be the canonical projection. 
A \emph{Weyl structure} of the parabolic geometry is a global 
smooth $G_0$--equivariant section $\si:\G_0\to\G$ of the projection $\pi$.
In particular, any Weyl structure provides a reduction of the principal bundle 
$\G\to M$ to the subgroup $G_0\subset P$.
For arbitrary parabolic geometry, Weyl structures always exist and any two 
Weyl structures $\si$ and $\hat \si$ differ by a $G_0$--equivariant mapping 
$\U:\ba_0\rightarrow \fp_+$ so that $\hat \si(u)= \si(u)\cdot \exp\U(u)$, 
for all $u \in \ba_0$. 
Since $\U$ is the frame form of a one--form on $M$, all Weyl structures form an affine 
space modelled over $\Om^1(M)$ and the relation above is simply written as 
$\hat \si = \si + \U$. 

Denote by $\om_i$ the $\fg_i$--component of the Cartan connection
$\om\in\Om^1(\G,\fg)$. 
The choice of the Weyl structure $\si$ defines the collection of
$G_0$--equivariant one--forms $\si^*\om_i\in\Om^1(\G_0,\fg_i)$.
The one--form $\si^*\om_0$ reproduces the fundamental vector fields of the
principal action of $G_0$ on $\G_0$, hence it defines a principal connection on $\G_0$
which we call the \emph{Weyl connection} of the Weyl structure $\si$.
The Weyl connection induces connections on all bundles associated to $\G_0$
and these are often called by the same name.
For any $i\ne 0$, the one--form $\si^*\om_i$ is strictly horizontal, hence it
descends to a one--form on $M$ with values in $\A_iM:=\A^iM/\A^{i+1}M$.
In particular, the whole negative part 
$\si^*\om_-=\si^*\om_{-k}\oplus\dots\oplus\si^*\om_{-1}$, which is called the
\emph{soldering form}, provides an identification of the tangent bundle $TM$ with the 
associated graded tangent bundle $\gr(TM)\cong\A_{-k}M\oplus\dots\oplus\A_{-1}M$.
The positive part $\si^*\om_+=\si^*\om_1\oplus\dots\oplus\si^*\om_k$ is called the
\emph{Rho--tensor} and denoted as $\Rho$.
The Rho--tensor is used to compare the Cartan connection $\om$ on $\G$ and the
principal connection on $\G$ extending the Weyl connection $\si^*\om_0$ from
the image of $\si:\G_0\to\G$.
By definition, $\Rho$ is a one--form on $M$ with values in $\A_1M\oplus\dots\oplus\A_kM$ 
and since this bundle is identified with $T^*M$, the Rho--tensor can be viewed as a section of
$T^*M\otimes T^*M$.

Among general Weyl structures, there are various specific subclasses.
We focus on the so--called normal Weyl structures which play some role in the
sequel. 
Normal Weyl structures are related to the notion of normal coordinates as follows.
Given a parabolic geometry $(p:\G\to M,\om)$ of type $(G,P)$ and a fixed
element $u\in\G$, the \emph{normal coordinates} at $x=p(u)$ is the local
diffeomorphism $\Phi_u$ from a neighborhood $U$ of $0\in\fg_-$ to a
neighborhood of $x\in M$, defined by $X\mapsto p(\Fl^{\om^{-1}(X)}_1(u))$.
(By $\om^{-1}(X)$ we denote the constant vector field on $\G$ corresponding to
$X$.)
Now, over the image $\Phi_u(U)\subset M$, there is a unique $G_0$--equivariant section 
$\si_u:\G_0\to \G$ such that $\Fl^{\om^{-1}(U)}_1(u)\subset\si_u(\G_0)$, which we
call the \emph{normal Weyl structure} at $x$.
Although the normal Weyl structure is indexed by $u\in\G$, it obviously
depends only on the orbit of $u\in p^{-1}(x)$ by  the action of $G_0$. 
If $\na$ and $\Rho$ is the corresponding affine connection and Rho--tensor,
respectively, then the normal Weyl structure $\si_u$ is characterized by the property
that for all $k\in\Bbb N$ the symmetrization of
$(\na_{\xi_k}\dots\na_{\xi_1}\Rho)(\xi_0)$ over all $\xi_i\in TM$ vanishes at
$x=p(u)$. 
Hence, in particular, $\Rho(x)=0$, cf.\ \cite[Theorem 3.16]{WS}.

\subsection{Automorphisms and symmetries}		\label{2.3}
An \emph{automorphism} of Cartan geometry $(\G\to M,\om)$ of type $(G,P)$ is a 
principal bundle automorphism $\va:\G\to\G$ such that $\va^*\om=\om$. 
It is well known that all automorphisms of (a connected component of) the homogeneous 
model $(G\to G/P,\mu)$ are just the left multiplications by elements of $G$.
Any $g\in G$ induces a base map $\ell_g:G/P\to G/P$ and it turns out that two
elements of $G$ have got the same base map if and only if they differ by an
element from the \emph{kernel} $K$ of  the pair $(G,P)$, which is the
maximal normal subgroup of $G$ contained in $P$.
(If $K$ is trivial then $G$ acts effectively on $G/P$.)
Moreover, the same characterization holds also for general Cartan geometries,
\cite[chapters 4 and 5]{S}.

In the cases of parabolic geometries, the kernel $K$ is always discrete 
and very often finite if not trivial.
An automorphism $\va:\G\to\G$ of parabolic geometry is then uniquely
determined by its base map $\underline\va:M\to M$  up to a smooth equivariant
function $\G\to K$ which has to be constant over connected components of $M$.

\begin{def*}
Let $(\ba \rightarrow M,\om)$ be a regular $|k|$--graded parabolic geometry, 
let $TM=T^{-k}M\supset\dots\supset T^{-1}M$ be the corresponding filtration 
of the tangent bundle, and let $x\in M$ be a point.
A \emph{local symmetry} of the parabolic geometry centered at $x$ 
is a locally defined diffeomorphism $s_x$ of a neighborhood of $x$ such that:
\begin{itemize} 
\item[(i)] $s_x(x)=x$, 
\item[(ii)] $T_x s_x|_{T_x^{-1}M}=-\operatorname{id}_{T_x^{-1}M}$, 
\item[(iii)] $s_x$ is covered by an automorphism of the parabolic geometry. 
\end{itemize}
If the local symmetry can be extended to a global symmetry on $M$, we just
speak about \emph{symmetry}.
The parabolic geometry is called (\emph{locally}) \emph{symmetric} if there is a
(local) symmetry at each point $x \in M$.
\end{def*}

Note that for $|1|$--graded parabolic geometries the restriction in the 
condition (ii) above is actually superfluous since $T^{-1}M=TM$.
Hence it can be shown that $s_x$ is involutive and $x$ is an isolated fixed point.
In this view, the definition above reflects the classical notion of affine 
locally symmetric spaces.
The main difference to the classical issues is that parabolic geometries are
not structures of first order, hence, in particular, the conditions above do not 
determine the symmetry uniquely.

Note also that the condition (ii) cannot be extended to the whole $TM$ in general:
Any symmetry is by definition covered by an automorphism of the parabolic
geometry, hence it has to preserve the underlying structure.
For instance, consider a parabolic contact geometry introduced in example
\ref{2.2}(2).
In particular, the underlying structure comprise of the contact distribution
$\D\subset TM$ and the non--degenerate Levi bracket $\L:\D\x\D\to TM/\D$.
If there was a map $s$ satisfying (i), (iii), and $T_x s=-\id_{T_xM}$, then 
for any $\xi,\eta\in\D_x$ it would hold
$s(\L(\xi,\eta)) =\L(s(\xi),s(\eta))=\L(\xi,\eta)$
and, simultaneously, 
$s(\L(\xi,\eta))=-\L(\xi,\eta)$,
which would contradict the non--degeneracy of $\L$.

\subsection{Symmetries of homogeneous models} 	\label{2.4}
Let $(G\rightarrow G/P,\mu)$ be the homogeneous model of a parabolic geometry
of type $(G,P)$ and let $G/P$ be connected.
As we mentioned in the beginning of \ref{2.3}, 
all automorphisms of the homogeneous model are just  the left 
multiplications by elements of $G$.
Next, an analog of the Liouville theorem states that any local automorphism
can be uniquely extended to a global one.
Hence if the homogeneous model is locally symmetric then it is symmetric.
By the transitivity of the action of $G$ on $G/P$ and the above
characterization of the automorphisms, in order to decide whether the homogeneous model 
is symmetric, it suffices to find a symmetry at the origin.
Due to the identification $T(G/P)\cong G\x_P(\fg/\fp)$ as in \ref{2.1}, the
previous task is equivalent to find an element in $P$ which acts as $-\id$ on
$\fg^{-1}/\fp\subset\fg/\fp$.
Since $P=G_0\rtimes P_+$ and $P_+$ acts trivially on $\fg^{-1}/\fp$, one is
actually looking for an element of $G_0$ acting as $-\id$ on $\fg_{-1}$.
Altogether, we have got the following general statement:
\begin{prop*}
  All symmetries of the homogeneous model $(G \rightarrow G/P, \mu)$ of a
  parabolic geometry of type $(G,P)$ centered at any point are parametrized 
  by the elements $g_0 \exp Z \in P$, where $Z\in\fp_+$ is arbitrary and 
  $g_0 \in G_0$ such that  $\Ad_{g_0}|_{\fg_{-1}}=-\id|_{\fg_{-1}}$. 
  In particular, if there is one symmetry at a point then there is an infinite
  amount of them.
\end{prop*} 

It is usually a simple exercise to find all elements from $G_0$ with
the property as above.
Note that different choices of the pair of Lie groups $(G,P)$ with the Lie algebras 
$\fp\subset\fg$ may lead to different amount of such elements.
This actually corresponds to the cardinality of the kernel $K$, as defined in \ref{2.3}.

\section{$|1|$--graded and Grassmannian locally symmetric spaces} \label{3}
Firstly we collect the facts on symmetries which hold for general $|1|$--graded
parabolic geometries.
Then we focus on Grassmannian locally symmetric spaces and provide a
discussion which is specific in that case.
In any case, the parabolic geometry in question is $|1|$--graded, so
the tangent map to a possible symmetry at $x\in M$ acts as $-\id$ on all of $T_xM$.
Hence the following fact is obvious and often used below:
\begin{lemma}			\label{lema}
  For a $|1|$--graded parabolic geometry on $M$, tensor field of odd degree which 
  is invariant with respect to a symmetry at $x\in M$ vanishes at $x$.
\end{lemma}

\subsection{General restrictions}		\label{3.2}
Following \cite[section 4]{WS}, we start with a bit of notation we use below.
Let $(\G\to M,\om)$ be a normal $|1|$--graded parabolic geometry,
let $\ka$ be the Cartan curvature, and let $\ka_H$ be its harmonic curvature.
Let $\si:\G_0\to\G$ be a Weyl structure, and let $\tau_i:=\si^*\om_i$ be the
corresponding Weyl forms as in \ref{2.5}.
Let us consider the curvature
$d\tau+\frac12[\tau,\tau]=\si^*\ka\in\Om^2(\G_0,\fg)$ and its
decomposition $T+W+Y$ according to the values in $\fg_{-1}\oplus\fg_0\oplus\fg_1=\fg$.
As before, $T,W$, and $Y$ is represented by a two--form on $M$ with values in
$\A_{-1}M\cong TM$, $\A_0M\cong \operatorname{End}_0(TM)$, and $\A_1M\cong T^*M$,
respectively.
By definition, $T=d\tau_{-1}+[\tau_{-1},\tau_0]$, hence it coincides with the
torsion of the affine connection $\na$ on $M$ induced by the Weyl connection
$\tau_0$. 
By lemma \ref{2.1}, this further coincides with the  homogeneous component of degree
one of the harmonic curvature $\ka_H$, hence it is independent of the choice of 
Weyl structure. 
Similarly, $W=d\tau_0+\frac12[\tau_0,\tau_0]+[\tau_{-1},\tau_1]$, 
where the first two summands represent just the curvature $R$ of $\na$.
Since $\tau_1=\Rho$, the last summand is rewritten as $\partial\Rho$, where
$(\partial\Rho)(\xi,\eta)=\{\xi,\Rho(\eta)\}+\{\Rho(\xi),\eta\}$, where $\{\ ,\
\}$ is the algebraic bracket on $\A M$ given by the Lie bracket in $\fg$.
Altogether,
\begin{equation}		\label{eq1}
  W=R+\partial\Rho
\end{equation}
and we call $W$ the \emph{Weyl curvature}.
If $T=0$ then  $W$ coincides by lemma \ref{2.1} with 
the homogeneous component of $\ka_H$ of degree two and so it is invariant 
with respect to the change of Weyl structure.

As an immediate application of lemma \ref{lema}  we have got the following:
\begin{prop*}
  If a $|1|$--graded parabolic geometry is locally symmetric then it is
  torsion--free. 
  In particular, any underlying Weyl connection is torsion--free.
\end{prop*}

Now we are going to find some information on the curvature of 
symmetric $|1|$--graded parabolic geometries.
If $\va$ is an automorphism of the parabolic  geometry then for 
arbitrary Weyl structure $\hat \si$ the pullback $\va^* \hat \si$ is again 
Weyl structure, hence $\va^*\hat \si =\hat \si + \U$, for some uniquely given 
one--form $\U$.
If $\va$ in addition covers some symmetry at $x$ then one checks that the 
Weyl structure $\si:=\hat \si + {1 \over 2} \U$ satisfies $\va^* \si = \si$ in the 
fiber over $x$. 
(Restricted to the fiber over $x$, the definition of $\si$ does not depend on 
$\hat \si$.) 
Next, let $\bar\si$ be the normal Weyl structure at $x$ which is uniquely determined by
$\si$ over $x$.
Since $\va^*\bar\si$ is again normal and by construction
$\va^*\bar\si=\bar\si$ over $x$, it has to coincide with $\bar\si$ on its domain.
Altogether, \cite[sections 9 and 10]{dis}:  
\begin{lemma*}
  Let $\va$ be an automorphism of $|1|$--graded parabolic geometry which covers 
  a local symmetry at a point $x$. Then
  \begin{enumerate}
  \item there is a Weyl structure which is invariant under $\va$ over the point $x$,
  \item there is a unique normal Weyl structure which is invariant under $\va$
  	over some neighborhood of $x$.   
  \end{enumerate}
\end{lemma*}

Let $\na$ be the affine connection corresponding to the normal Weyl
structure on a neighborhood of $x$ as above and let $T,R,W$, and $\Rho$  be its torsion,
curvature, Weyl curvature, and Rho--tensor, respectively.
By construction, the connection $\na$ is invariant with respect to the local
symmetry at $x$ and by the previous Proposition, $T=0$.
Similarly, $\na W$ is also invariant under the symmetry at $x$, however, it is a tensor 
field of degree five which vanishes at $x$ by lemma \ref{lema}.
Since $\na$ is normal at $x$, $\Rho$ vanishes at $x$, hence from the equation \eqref{eq1}
we conclude that also $\na R=0$ at $x$:
\begin{thm*}
  Suppose there is a local symmetry $s_x$ of a $|1|$--graded parabolic geometry
  centered at $x$.
  Then, on a neighborhood of $x$, there exists a torsion--free Weyl connection 
  $\na$  which is invariant under $s_x$ and whose Rho--tensor vanishes at $x$.
  Consequently, $\na R$ vanishes at $x$. 
\end{thm*}

\subsection{Almost Grassmannian structures}		\label{3.3}
The notion of almost Grassmannian structure on a smooth manifold generalizes
the geometry of Grassmannians.
The \emph{Grassmannian of type $(p,q)$} is the space
$\Gr(p,\R^{p+q})$ of $p$--dimensional linear subspaces in the
real vector space $\R^{p+q}$.
It is well known that, for any $E\in\Gr(p,\R^{p+q})$,  the tangent space of 
$\Gr(p,\R^{p+q})$ in $E$ is identified with $E^*\otimes(\R^{p+q}/E)$, the space of 
linear maps from $E$ to the quotient.
In particular, the dimension of $\Gr(p,\R^{p+q})$ is $pq$.
Note that $\Gr(1,\R^{1+q})$ is the real projective space $\R\Bbb P^q$, so
we always assume $p>1$ hereafter.
Under this assumption, the Grassmannian $\Gr(p,\R^{p+q})$ may be considered as
the space of $(p-1)$--dimensional projective subspaces in $\R\Bbb
P^{p+q-1}=\P(\R^{p+q})$. 
The Lie group $\hat G:=PGL(p+q,\R)$ acts transitively (and effectively) on 
$\Gr(p,\R^{p+q})$ and the stabilizer of a fixed element is the parabolic
subgroup $\hat P$ as described below.
The Grassmannian of type $(p,q)$ is then the homogeneous model of the
parabolic geometry of type $(\hat G,\hat P)$.

The Lie algebra of $\hat G$ is $\hat\fg=\frak{sl}(p+q,\R)$ and let us consider
its grading which is schematically described by the block decomposition
$$ 
  \pmat{\hat \fg_0 & \hat \fg_1 \\ \hat \fg_{-1} & \hat \fg_0}
$$
with blocks of sizes $p$ and $q$ along the diagonal. 
In particular, $\hat \fg_{-1}\cong \R^{p*}\otimes\R^q$, $\hat \fg_0 \cong 
\frak s(\frak{gl}(p,\R) \oplus \frak{gl}(q,\R))$, and $\hat \fg_{1}\cong 
\R^p\otimes\R^{q*}$.
The parabolic subgroup $\hat P\subset\hat G$ is represented by block upper 
triangular matrices with the Lie algebra $\hat \fp = \hat \fg_0 \oplus\hat\fg_1$,
the subgroup $\hat G_0$ corresponds then to the block diagonal matrices in $\hat
P$.

An \emph{almost Grassmannian structure of type $(p,q)$} on a smooth
manifold $M$ is defined to be a $|1|$--graded parabolic geometry of type 
$(\hat G,\hat P)$ where the groups are as above.
The underlying structure on $M$ is equivalent to the choice of auxiliary
vector bundles $E\to M$ and $F\to M$ of rank $p$ and $q$, respectively, and an
isomorphism $E^*\otimes F\to TM$.
Note that a different choice of the Lie group to the Lie algebra
$\hat\fg=\frak{sl}(p+q,\R)$ gives rise to an additional structure on $M$.
In particular, the usual choice for $\hat G$ to be $SL(p+q,\R)$ leads to a preferred 
trivialisation of $\wedge^pE\otimes \wedge^qF$ which is often supposed in the
literature. 
In contrast to the previous choice, this group has got a non--trivial 
center provided $p+q$ is even.

\subsection{Grassmannian locally symmetric spaces}	\label{3.4}
When we speak about a \emph{Grassmannian (locally) symmetric space}, we mean a smooth 
manifold with an almost Grassmannian structure which is (locally) symmetric in 
the sense of \ref{2.3}.
For technical reasons we always assume the almost Grassmannian structure is
represented by a normal parabolic geometry of type $(\hat G,\hat P)$, which is
uniquely determined by the underlying structure up to isomorphism.
Note that by Proposition \ref{3.2}, any Grassmannian locally symmetric space
admits  a torsion--free affine connection preserving the structure, hence the
almost Grassmannian structure is actually Grassmannian, i.e.\ it is integrable in 
the sense of $G$--structures.

According to Proposition \ref{2.4}, it is an easy exercise to decide whether
the homogeneous model $\hat G/\hat P$ is symmetric or not. 
A direct calculation shows that, \cite[section 7]{dis}:
\begin{prop*}
  The homogeneous model of almost Grassmannian structures of type $(p,q)$ is
  always symmetric.
\end{prop*}
An explicit description of the harmonic curvature in individual cases yields
that if $p>2$ and $q>2$ then this has got two components in homogeneity one,
i.e.\ two torsions.
Hence by Propositions \ref{3.2} and lemma \ref{2.1} we conclude that, \cite[Corollary
3.2]{ja}:
\begin{prop*}
  Grassmannian locally symmetric space of type $(p>2,q>2)$
  is flat, i.e.\ locally isomorphic to the homogeneous model.
\end{prop*}

Hence the only non--flat almost Grassmannian structures which can carry 
(local) symmetries are of type $(p,q)$ where $p$ or $q$ is 2; this is always
supposed hereafter.
In all these cases, the harmonic curvature has two components which are mostly
of homogeneity one and two.
(Note that the extremal case $p=q=2$ has a specific feature, namely, there are
two components of homogeneity two.
Moreover, an almost Grassmannian structure of type $(2,2)$ is equivalent to a 
conformal pseudo--Riemannian structure of the split signature, \cite[section 3.5]{C}.)
Since the torsion part vanishes for any Grassmannian locally symmetric space, 
the remaining component of homogeneous degree two corresponds to the Weyl curvature 
which is then the only obstruction to the local flatness of the structure.

From  \ref{3.2} we know that the existence of a local symmetry at a point $x$ 
yields some restriction on the Weyl curvature at that point.
This heavily forces the freedom for another possible symmetries at $x$:
Suppose there are two different local symmetries at $x$
which are covered by $\va_1$ and $\va_2$. 
Let $\si_1$ and $\si_2$ be the Weyl structures associated to $\va_1$ and
$\va_2$ by lemma \ref{3.2} and let $\U$ be the one--form such that
$\si_2=\si_1+\U$. 
In this way, two different symmetries at $x$ defines an element $\U_x$ in 
$T^*_xM\cong E_x\otimes F^*_x$, which turns out to be non--zero.
With this notation, it can be proved the following, \cite[section 4.3]{ja-dga07}:
\begin{thm*}			\label{twosym}
  Let $M$ be a smooth manifold with an almost Grassmannian structure of type $(2,q)$ or 
  $(p,2)$.
  If there are two different local symmetries centered at a point $x\in M$
  and the corresponding covector $\U_x$ constructed above has maximal rank
  then the Weyl curvature vanishes at $x$ and, consequently, the Cartan
  curvature vanishes at $x$.
\end{thm*}
(Note that the Theorem is originally formulated with respect to a one--form
which was constructed in a different way.
However, it is easy to check the two one--forms coincide up to a non--zero 
multiple at $x$.)

\section{The example}
In this section we present an example of non--flat homogeneous Grassmannian 
symmetric space. 
By the previous results, this has to be necessarily either of type 
$(2,q)$ or of type $(p,2)$.
Our example is of the former type and it comes as a particular application of  more 
general constructions from \cite{CZ} and \cite{C} where we refer for details.
The section concludes with a discussion on an invariant connection preserving
the Grassmannian structure on the constructed space.

\subsection{Path geometry of chains}		\label{4.2}
Let $(\G\to M,\om)$ be a contact parabolic geometry of type $(G,P)$ and let
$\fg$ be the corresponding Lie algebra with the contact grading as in example
\ref{2.2}(2).
The 1--dimensional subspace $\fg_{-2}\subset\fg_-$ gives rise to a family of
distinguished curves on $M$ which are called the \emph{chains} and which play
a crucial role in the sequel.
More specifically, chains are defined as projections of flow lines of constant vector 
fields on $\G$ corresponding to non--zero elements of $\fg_{-2}$.
Equivalently, using the notion of development of curves, chains are the curves which
develop to model chains in the homogeneous model $G/P$.
The latter curves passing through the origin are the curves of type 
$t\mapsto b\exp(tX)P$, for $b\in P$ and $X\in\fg_{-2}$.
As non--parametrized curves, chains are uniquely determined by a tangent direction 
in a point, \cite[section 4]{CSZ}.
By definition, chains are defined only for directions which are transverse to
the contact distribution $\D\subset TM$.
In classical terms, the family of chains defines a path geometry on
$M$ restricted, however, only to the directions transverse to $\D$.

A \emph{path geometry} on $M$ is equivalent to a decomposition
$\Xi=E\oplus V$ of the tautological subbundle $\Xi\subset T\P TM$  
where $V$ is the vertical subbundle of the obvious projection 
$\P TM\to M$ and $E$ is a fixed transversal line subbundle.
(Given such a decomposition, the paths in the family are the projections of integral 
submanifolds of the distribution $E$.)
Lie bracket of vector fields behave specifically with respect to the 
decomposition above and it turns out this structure on $\P TM$ can be described as a 
parabolic geometry 
of type $(\tG,\tP)$, where
$\tG=PGL(m+1,\R)$, $m=\dim M$, and $\tP$ is the parabolic subgroup as follows.
Let us consider the grading of $\tg=\frak{sl}(m+1,\R)$ which is schematically 
described by the block decomposition with blocks of sizes 1, 1, and $m-1$ along 
the diagonal:
$$
  \pmat{\tg_0&\tg^E_1&\tg_2\\\tg^E_{-1}&\tg_0&\tg^V_1\\\tg_{-2}&\tg^V_{-1}&\tg_0}.
$$
Then $\tp:=\tg_0\oplus\tg_1\oplus\tg_2$ is  a parabolic subalgebra of $\tg$ and 
$\tP\subset\tG$ is the subgroup represented by block upper triangular
matrices so that its Lie algebra is $\tp$.
As usual, the parabolic geometry associated to the path geometry on $M$ is uniquely 
determined (up to isomorphism) provided we consider it is regular and 
normal in the sense of \ref{2.1}.

Back to the initial setting, given a contact manifold $M$ with a parabolic contact 
geometry of type $(G,P)$, the path geometry of chains gives rise to a parabolic 
geometry of type $(\tilde G,\tP)$ restricted to the open subset $\tilde M\subset\P 
TM$ consisting of all lines which are transverse to the contact distribution 
$\D\subset TM$.
Let $Q\subset P$ be the subgroup which stabilizes the subspace
$\fg_{-2}\subset\fg_-$ under the action of $P$ on $\fg_-$ induced from the adjoint
action on $\fg$; the Lie algebra of $Q$ is evidently $\fq=\fg_0\oplus\fg_2$.
By \cite[lemma 2.2]{CZ}, the space $\tilde M$ of all lines in $TM$ transverse to 
$\D$ is identified with the orbit space $\G/Q$.

Altogether, for a parabolic contact structure on $M$ given  by a regular and normal 
parabolic geometry $(\G\to M,\om)$ of type $(G,P)$, let $(\tilde\G\to\tilde
M,\tom)$ be the regular normal parabolic geometry of type $(\tG,\tP)$ corresponding 
to the path geometry of chains.
Due to the identification $\tilde M\cong\G/Q$, the couple $(\G\to\tilde M,\om)$
forms a Cartan (but not parabolic) geometry of type $(G,Q)$.
In some cases, the two Cartan geometries over $\tilde M$ can be directly
related by a pair of maps $(i:Q\to\tP,\al:\fg\to\tg)$ so that
$\tilde\G\cong\G\x_Q\tP$ and $j^*\tilde\om=\al\o\om$, where $j$ is the canonical 
inclusion $\G\hookrightarrow\G\x_Q\tP$.
The two maps $(i,\al)$ has to be compatible in some strong sense by the
equivariancy of $j$ and the fact that both $\om$ and $\tilde\om$ are Cartan 
connections, \cite[Proposition 3.1]{CZ}.
On the other hand, any pair of maps $(i,\al)$ which are compatible in the above 
sense gives rise to a functor from Cartan geometries of type $(G,Q)$ to 
Cartan geometries of type $(\tG,\tP)$ and there is a perfect control over the natural 
equivalence of functors associated to different pairs.
For what follows, we need to understand the effect of such construction on the
curvature of the induced Cartan geometry.
In particular, \cite[Proposition 3.3]{CZ} shows that:
\begin{lemma*}		\label{l1}
  Let a flat Cartan geometry of type $(G,Q)$ be given.
  Then the Cartan geometry of type $(\tG,\tP)$ induced by the pair $(i,\al)$
  is flat if and only if $\al$ is a homomorphism of Lie algebras.
\end{lemma*}

\subsection{Correspondence spaces and twistor spaces}		\label{4.3}
Below we enjoy an application of another general construction relating
parabolic geometries of different types, namely, the construction of 
correspondence spaces and twistor spaces in the sense of 
\cite{Cor} or \cite{C}, to which we refer for all details.

Let $\tG$ be a semisimple Lie group and $\tP_1\subset\tP_2\subset\tG$ parabolic
subgroups. 
If a parabolic geometry $(\tilde\G\to\tilde N,\tom)$ of type $(\tG,\tP_2)$ on
a smooth manifold $\tilde N$ is given, then the \emph{correspondence space} 
of $\tilde N$ corresponding to the subgroup $\tP_1\subset\tP_2$ is defined as the 
orbit space $\mathcal{C}\tilde N:=\tilde\G/\tP_1$.
The couple $(\tilde\G\to\mathcal{C}\tilde N,\tom)$ forms a parabolic geometry
of type $(\tG,\tP_1)$.
Let $\mathcal{V}\subset T\mathcal{C}\tilde N$ be the vertical subbundle of 
the natural projection $\mathcal{C}\tilde N\to\tilde N$.
Then easily, $i_\xi\tilde\ka=0$ for any $\xi\in \mathcal{V}$, where 
$\tilde\ka$ is the Cartan curvature of $\tom$.
Note that $\mathcal{V}$ corresponds to the $\tP_1$--invariant subspace
$\tp_2/\tp_1\subset\tg/\tp_1$ under the identification $T\mathcal{C}\tilde
N\cong\tilde\G\x_{\tP_1}(\tg/\tp_1)$.

Conversely, given a parabolic geometry $(\tilde\G\to\tilde M,\tom)$ of type
$(\tG,\tP_1)$, let $\tilde\ka$ be its Cartan curvature, and let 
$\mathcal{V}\subset T\tilde M$ be the distribution corresponding to 
$\tp_2/\tp_1\subset\tg/\tp_1$.
Then, locally, $\tilde M$ is a correspondence space of a parabolic geometry 
of type $(\tG,\tP_2)$   if and only if $i_\xi\tilde\ka=0$ for all 
$\xi\in \mathcal{V}$, \cite[Theorem 3.3]{C}.
Note that this condition in particular implies the distribution $\mathcal{V}$ is
integrable and the local leaf space of the corresponding foliation is called
the \emph{twistor space}.
The parabolic geometry of type $(\tG,\tP_2)$ is locally formed over the 
corresponding twistor space.

Note that the present considerations does not restrict only to parabolic
geometries, as we actually partially observed in the previous subsection.
Still, for parabolic geometries the constructions above are always compatible with 
the normality condition. 
Concerning the regularity, this is not true in general, but an efficient control 
of this condition is usually very easy.
Dealing with a regular normal parabolic geometry of type $(\tG,\tP_1)$, 
let $\tilde\ka_H$ be the harmonic curvature and let $\mathcal{V}\subset
T\tilde M$ be as above. 
Then there is the following useful simplification of the previous characterization of
correspondence spaces, \cite[Proposition 3.3]{C}:
If $i_\xi\tilde\ka_H=0$ for all $\xi\in \mathcal{V}$ then $i_\xi\tilde\ka=0$ for all 
$\xi\in \mathcal{V}$.

\subsection*{Example}
Let $(\tilde\G\to\P TM,\tom)$ be the Cartan geometry associated to a path
geometry on $M$, i.e.\ a parabolic geometry of type $(\tG,\tP)$ with the
notation as in \ref{4.2}.
Let $\hat P\subset\tG$ be the subgroup (containing $\tP$ and) consisting of
block upper triangular matrices with Lie algebra
$\hat\fp=\tg_{-1}^E\oplus\tp$ according to the description above.
Note that the underlying structure of a parabolic geometry of type
$(\tG,\hat P)$ is just the Grassmannian structure of type $(2,q)$, where
$q=\dim M-1$, cf.\ the definition in \ref{3.3}.
The distribution in $T\P TM$ corresponding to the linear subspace
$\hat\fp/\tp\subset\tg/\tp$ is just the line subbundle $E$ determined by the
path geometry on $M$, in particular this is always involutive.
Hence the corresponding local twistor space $\tilde N$ coincides locally with the space 
of paths of the path geometry.
From the above characterization of correspondence spaces and the explicit description of the 
irreducible components of the harmonic curvature $\tilde\ka_H$ of the Cartan
connection $\tom$, it follows that \cite[example 3.4]{C}:
\begin{lemma*}		\label{l4}
  Let $(\tilde\G\to\P TM,\tom)$ be a Cartan geometry of type $(\tG,\tP)$ and
  let $\tilde N$ be the local twistor space as above.
  Then the Cartan geometry on $\P TM$ descends to a Grassmannian 
  structure on $\tilde N$ if and only if $\tom$ is torsion--free.
\end{lemma*}

\subsection{Applications}			\label{4.4}
Now, the promised example of a Grassmannian symmetric space appears as an
application of the general principles described in previous paragraphs.
We are going to start with the model Lagrangean contact structure, however
the analogous ideas work for another parabolic contact structures as well. 
This is discussed in remark \ref{4.6}(4) where we also highlight the differences.

A \emph{Lagrangean contact structure} on a smooth manifold $M$ of odd dimension
$m=2n+1$ consists of a contact distribution $\D\subset TM$ with a decomposition
$\D=L\oplus R$ such that the subbundles are isotropic with respect to the Levi 
bracket $\L:\D\x \D\to TM/\D$.
Lagrangean contact structure is an instance of parabolic contact structure
corresponding to the contact grading of simple Lie algebra
$\fg=\frak{sl}(n+2,\R)$, which is schematically indicated by the following
block decomposition with blocks of sizes 1, $n$, and 1 along the diagonal:
$$
\pmat{\fg_0&\fg_1^L&\fg_2\\ \fg_{-1}^L&\fg_0&\fg_1^R\\
\fg_{-2}&\fg_{-1}^R&\fg_0}.
$$
As in general, the subspace $\fg_{-1}$ defines the contact distribution,
however now it is split as $\fg_{-1}=\fg_{-1}^L\oplus\fg_{-1}^R$ such that
this splitting is invariant under the adjoint action of $\fg_0$
and the subspaces $\fg_{-1}^L$ and $\fg_{-1}^R$ are isotropic  with
respect to the restricted Lie bracket $[\ ,\  ]:\fg_{-1}\x\fg_{-1}\to\fg_{-2}$.
Let $G=PGL(n+2,\R)$  be the Lie group with Lie algebra $\fg$ and let $P\subset
G$ be the subgroup represented by block upper triangular matrices with the Lie algebra
$\fp=\fg_0\oplus\fg_1\oplus\fg_2$.
The homogeneous space $G/P$ is identified with $\P T^*\R\Bbb P^{n+1}$, 
the projectivized cotangent bundle of real projective space of dimension $n+1$,
and the model Lagrangean contact structure on $\P T^*\R\Bbb P^{n+1}$ is
induced from the flat projective structure on $\R\Bbb P^{n+1}$.
Note that this correspondence is just another instance of the correspondence space 
constructions from \ref{4.3}, see \cite[section 4.1]{Cor}.

\smallskip  
Now, put $M=G/P$ and follow the construction from \ref{4.2}:

(1) The subset $\tilde M=\P_0TM$, consisting of all lines in $TM$ which are
transverse to the contact distribution $\D$, is identified with the 
homogeneous space $G/Q$.

(2) The flat parabolic geometry $(G\to G/P,\mu)$ of type $(G,P)$, for $\mu$ being the 
Maurer--Cartan form on $G$, defines the flat Cartan geometry $(G\to G/Q,\mu)$
of type $(G,Q)$.

(3) The latter induces the parabolic geometry $(G\x_Q\tP\to G/Q,\tom)$ of type $(\tG,\tP)$ 
via the pair of maps $(i,\al)$ which are explicitly given in \cite[section 3.5]{CZ}.
Note that $\al:\fg\to\tg$ is not a homomorphism of Lie algebras, hence by lemma \ref{l1}
the induced Cartan geometry is not flat.
By \cite[section 3.6]{CZ}, this is the unique regular normal parabolic geometry 
associated to the path geometry of chains:
\begin{lemma*}		\label{l5}
  Let $(G\to G/Q,\mu)$ be the flat Cartan geometry of type $(G,Q)$ and let
  $(i,\al)$ be the pair of maps as in step (3) above.
  Then the induced parabolic geometry $(G\x_Q\tP\to G/Q,\tom)$
  is a non--flat torsion--free (and hence regular) normal parabolic 
  geometry of type $(\tG,\tP)$.
\end{lemma*}

(4) Finally, let $\tilde N$ be the space of all chains on $M=G/P$, understood 
as non--parametrized curves as above.
By definition, this is a locally defined  leaf space of the foliation of
$\tilde M$ corresponding to the distribution $E$ as in \ref{4.3}.
In this model case, $\tilde N$ is a homogeneous space and it turns out to be
a Grassmannian symmetric space which is \emph{not} flat, i.e.\ not locally 
isomorphic to the homogeneous model $\tG/\hat P$:
\begin{thm*}					\label{th1}
  Let $M=G/P$ be the model Lagrangean contact structure.
  Then the space $\tilde N$ of all chains in $M$ is a non--flat homogeneous Grassmannian 
  symmetric space of type $(2,q)$, where $q=\dim M-1$.
\end{thm*}
\begin{proof}
  Almost everything follows immediately from the previous profound
  preparation:

  Lemmas \ref{l5} and \ref{l4} yield that $\tilde N$ is endowed with a Grassmannian 
  structure and the fact that the induced Cartan geometry of
  type $(\tG,\tP)$ on $\tilde M=G/Q$ has a non--trivial curvature implies the
  curvature of the corresponding Cartan geometry on the twistor space $\tilde
  N$ is non--trivial as well.
  Since $\tilde M=\P_0TM$ is a homogeneous space and any chain is uniquely
  determined by an element of $\tilde M$, the group $G$ acts transitively on
  the space $\tilde N$ of all chains.
  Let $H\subset G$  be the stabilizer of the chain
  $\exp tX\cdot P$, $X\in\fg_{-2}$, passing through the origin in
  $M=G/P$.
  An easy direct computation shows that $H$ is the subgroup consisting of
  block matrices in $G$ so that its Lie algebra is
  $\fh=\fg_{-2}\oplus\fg_0\oplus\fg_2$, i.e.\ $H=\exp\fg_{-2}\ltimes Q$.
  Altogether, $\tilde N\cong G/H$ and consequently $T\tilde N\cong G\x_H(\fg/\fh)$.
  
  By the very construction, elements of $G$ act as automorphisms of the
  induced Cartan geometry on $\tilde M$ and since the quotient $\hat P/\tP$ 
  is obviously connected, these descend to automorphisms of the Grassmannian 
  structure on  $\tilde N$ by \cite[remark 2.4]{Cor}.
  In order to show there is a symmetry at any point of $\tilde N$, it suffices
  to find an element in $H$ which acts as $-\id$ on $\fg/\fh$.
  However, this is rather easy task and after a while of calculation one shows
  that the block matrix
  $$
    \pmat{-1&0&0\\0&\Bbb I_n&0\\0&0&-1}
  $$
  represents the unique element with this property.
\end{proof}

\subsection*{Remark}
In the proof above we have constructed a global symmetry of the Grassmannian
structure at the origin of $\tilde N=G/H$, which leads to a distinguished symmetry 
at each point. 
Any such symmetry is represented by an element of $G$ and it will be called
the \emph{$G$--symmetry}.
Any $G$--symmetry primarily defines an automorphism of the Lagrangean contact structure 
on $M=G/P$ and, easily, this is a symmetry on $M$ in the sense of \ref{2.3}.
Since the parabolic contact structure on $M$ is flat, there is a lot of
symmetries at any point, but only one of them induces a symmetry on
$\tilde N$.
On the other hand, apart from the $G$--symmetry, there may be another local
symmetries at any point of $\tilde N$.
However, according to Theorem \ref{twosym}, all the possible symmetries may 
differ from the $G$--symmetry in a very restricted sense.
More specifically, by the homogeneity of the induced Grassmannian structure on
$\tilde N$, the corresponding Weyl curvature is nowhere vanishing, hence the
covector $\U_x$ from Theorem \ref{twosym} measuring the difference of two
symmetries at $x$ must be of rank one.

\subsection{Invariant connection}		\label{4.5}
Let us conclude by a discussion on an affine connection on
$\tilde N$ which preserves the Grassmannian structure and which is invariant with respect 
to some symmetries.
Note that the following statement can be seen as an instance of 
\cite[Theorem 1]{B} which deals with invariant connections on reductive
homogeneous spaces with a compatible additional structure of general $|1|$--graded 
parabolic  geometry.
Of course, in our specific setting we can approach the result in a more direct
way.
\begin{thm*}
  Let $\tilde N=G/H$ be the space of chains of model Lagrangean contact
  structure on $M=G/P$.
  Then there is a $G$--invariant torsion--free affine connection $\tilde\na$ on 
  $\tilde N$ preserving the Grassmannian structure.
\end{thm*}
\begin{proof}
  As we know from \ref{4.3}, 
  the construction of twistor space and the corresponding Cartan geometry is very local 
  in nature.
  However, in our model case, there is the global surjective submersion 
  $p:\tilde M=G/Q\to G/H=\tilde N$ to the twistor space.
  The principal $\tP$--bundle $\pi:G\x_Q\tP\to\tilde M$ is defined according to the
  Lie group homomorphism $i:Q\to\tP$, which we referred to in step (3) in \ref{4.4}.
  The homomorphism $i$ can be extended to a homomorphism $\hat i:H\to\hat P$ so that 
  $\hat i'=\al|_\fh$,  the total space of 
  the principal bundle $G\x_Q\tP\to\tilde M$ is identified with $G\x_H\hat P$, and
  the composition $p\o\pi:G\x_H\hat P\to\tilde N$ is a principal $\hat P$--bundle.
  The properties of the Cartan connection $\tom$ as above, namely the torsion freeness, 
  yield the couple $(G\x_H\hat P\to\tilde N,\tom)$ is a parabolic geometry of type 
  $(\tG,\hat P)$ and $\tilde M$ is the correspondence space for $\tP\subset\hat P$.

  Now, let $\hat G_0$ be the Lie subgroup of $\tG$ as in \ref{3.3}.
  Namely, $\hat G_0$ is represented by block diagonal matrices with the Lie algebra 
  $\hat\fg_0=\tg_{-1}^E\oplus\tg_0\oplus\tg_1^E$, i.e.\ the reductive part of $\hat\fp$.
  Since $\hat i:H\to\hat P$ is a homomorphism of Lie groups, $\hat
  i'=\al|_\fh$, and $\al(\fh)\subset\hat\fg_0$,  it follows 
  that $\hat i(H)\subset\hat G_0$.
  This gives rise to the principal $\hat G_0$--bundle  $G\x_H\hat G_0\to\tilde N$,
  which is a distinguished reduction of $G\x_H\hat P\to\tilde N$ to the structure
  group $\hat G_0\subset\hat P$.
  In terms of \ref{2.5}, this is a Weyl structure, which is evidently $G$--invariant.
  Finally, let $\tilde\na$ be the affine connection on $\tilde N$ induced by
  the Weyl structure above.
  By construction, $\tilde\na$ is $G$--invariant and torsion--free and,
  by general principles, it preserves the underlying geometric structure on $\tilde N$. 
\end{proof}

In remark \ref{4.4} we have defined the notion of $G$--symmetry at $x\in\tilde N$.
Since any $G$--symmetry is induced by an element of $G$, the connection
$\tilde\na$ constructed above is invariant with respect to all
$G$--symmetries.
Hence:
\begin{cor*}
  Let $\tilde\na$ be the affine connection on $\tilde N=G/H$ from the Theorem above. 
  Then $(\tilde N,\tilde\na)$ is an affine symmetric space in the classical
  sense, whose unique  symmetry at each point is the $G$--symmetry.
\end{cor*}

\subsection{Final remarks}			\label{4.6}
(1) For a reader's convenience and better orientation in the text, 
let us gather all the relations we have discussed from \ref{4.4} till now into 
the following picture:
$$
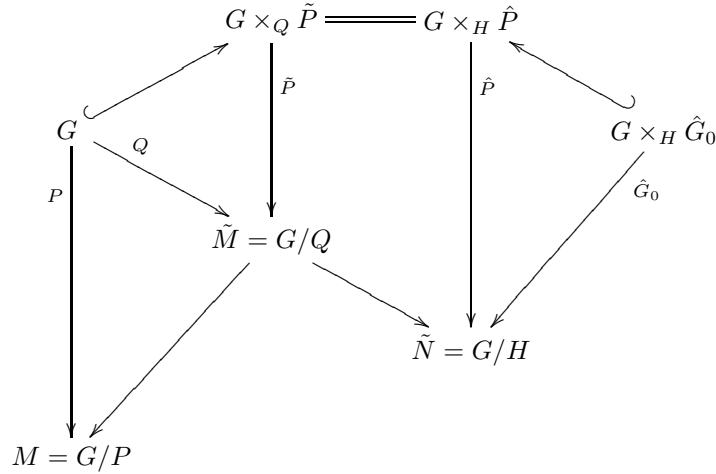
\xymatrix{
 & & G\x_Q\tP \ar[dd]^(.3){\tP} \ar@{=}[r] & G\x_H\hat P \ar[ddd]^(.2){\hat P} & \\
 & G\ \ar@{^{(}->}[ur] \ar[dr]^(.3)Q \ar[ddd]_(.2)P & & & 
 	G\x_H\hat G_0 \ar@{_{(}->}[ul] \ar[ddl]^(.2){\hat G_0} \\
 & & \tilde M=G/Q \ar[ddl] \ar[dr] & & \\
 & & & \tilde N=G/H & \\
 & M=G/P & & &
 }
$$

(2) Note that $\tilde N=G/H$ is a reductive homogeneous space, namely, the 
reductive decomposition is $\fg=\frak{n}\oplus\fh$ where 
$\fh=\fg_{-2}\oplus\fg_0\oplus\fg_2$ as above and $\frak{n}=\fg_{-1}\oplus\fg_1$.
In particular, restriction of the Cartan--Killing form of $\fg$ to $\frak{n}$
gives rise to a $G$--invariant (pseudo--)Riemannian metric on $\tilde N$ whose
Levi--Civita connection is the canonical $G$--invariant affine connection on 
$G/H=\tilde N$.
Since both this canonical connection and the connection from Theorem \ref{4.5}
are invariant with respect to the $G$--symmetry at each point, so is their
difference tensor, which  has to vanish by lemma \ref{lema}. 
Hence the two connections do actually coincide. 

(3)
Note also that for the lowest possible dimension of $M$, i.e.\ 3, the dimension of
$\tilde N$ is 4 and the induced almost Grassmannian structure is of type
$(2,2)$.
As we mentioned in \ref{3.4}, this structure on $\tilde N$ is equivalent to
the conformal pseudo--Riemannian structure of split signature.
Hence the constructions above yield to an example of non--flat conformal symmetric 
space in this specific signature.

(4) 
Note finally that the procedure of \ref{4.4} and \ref{4.5} can be applied to any parabolic contact
geometry, however, the resulting structure on the space of chains may differ.
For instance, starting with the flat projective contact structure, it turns
out that the associated path geometry of chains is locally flat, hence
it descends to a locally flat almost Grassmannian structure on the space of chains.
On the other hand, the story for CR structures of hypersurface type is
completely parallel to that in the Lagrangean contact case.
Of course, understanding the behavior for another parabolic contact structures is 
on the order of the day.


\end{document}